\documentclass[11pt,a4paper]{amsart}
\usepackage{graphicx,multirow,array,amsmath,amssymb,enumitem}

\newtheorem{theorem}{Theorem}

\newtheorem{lemma}[theorem]{Lemma}

\newtheorem*{one-to-one conversion}{One-to-one conversion}

\begin{document}

\title{State matrix recursion method and monomer--dimer problem}

\author[S. Oh]{Seungsang Oh}
\address{Department of Mathematics, Korea University, Seoul 02841, Korea}
\email{seungsang@korea.ac.kr}

\thanks{Mathematics Subject Classification 2010: 05A15, 05B45, 05B50, 82B20, 82D60}
\thanks{This work was supported by the National Research Foundation of Korea(NRF) grant funded
by the Korea government(MSIP) (No. NRF-2017R1A2B2007216).}

\maketitle

\begin{abstract}
The exact enumeration of pure dimer coverings on the square lattice was obtained by
Kasteleyn, Temperley and Fisher in 1961. 
In this paper, we consider the monomer--dimer covering problem (allowing multiple monomers)
which is an outstanding unsolved problem in lattice statistics.
We have developed the state matrix recursion method 
that allows us to compute the number of monomer--dimer coverings
and to know the partition function with monomer and dimer activities.
This method proceeds with a recurrence relation of so-called state matrices of large size.
The enumeration problem of pure dimer coverings and dimer coverings with single boundary monomer
is revisited in partition function forms.
We also provide the number of dimer coverings with multiple vacant sites.
The related Hosoya index and the asymptotic behavior of its growth rate are considered.
Lastly, we apply this method to the enumeration study of domino tilings of Aztec diamonds 
and more generalized regions, so-called Aztec octagons and multi-deficient Aztec octagons.
\end{abstract}

\section{Introduction} \label{sec:intro}

The monomer--dimer problem is one of simplicity of definition, but famous unsolved problem,
and has a long and glorious history.
The monomer--dimer system has been used as a model of a physical system~\cite{FR, Ka2},
but primarily it is interesting as the matching counting problem in combinatorics~\cite{LP}.
While it is known that it does not exhibit a phase transition~\cite{HL},
there have been only limited closed-form results.

It gained momentum in 1961 when Kasteleyn~\cite{Ka1} and Temperley and Fisher~\cite{Fi,TF} 
found the exact solution of the enumeration of pure dimer coverings (i.e., no monomers).
Pure dimer coverings are often considered as perfect matchings or domino tilings.
The total number of pure dimer coverings in the $m \! \times \! n$ square lattice with even $mn$ is
$$\prod^m_{j=1} \prod^n_{k=1} \left| 2 \cos(\frac{\pi j}{m+1}) + 2 i \cos(\frac{\pi k}{n+1}) \right|^{\frac{1}{2}}.$$

In 1974, Temperley~\cite{Te} found an intriguing bijection  
between spanning trees of the $m \! \times \! n$ square lattice
and pure dimer coverings in the $(2m \! + \! 1) \! \times \! (2n \! + \! 1)$ square lattice with a corner removed.
This offers an alternate approach to the vertex vacancy problem.
Recently, Tzeng and Wu~\cite{TW} used Temperley bijection
to enumerate dimer coverings with a fixed single monomer on the boundary.

The purpose of this paper is to introduce a method 
for the enumeration of monomer--dimer coverings (allowing multiple monomers),
that is called the {\em state matrix recursion method\/}.
More precisely, it provides a recursive formula of state matrices to give the partition function
with respect to monomer and dimer activities.
A typical example of a monomer--dimer covering
in the $m \! \times \! n$ square lattice is drawn in Figure~\ref{fig:MD}.
In Section~\ref{sec:problem}, 
we state several monomer--dimer problems that are considered in this paper.

\begin{figure}[h]
\includegraphics{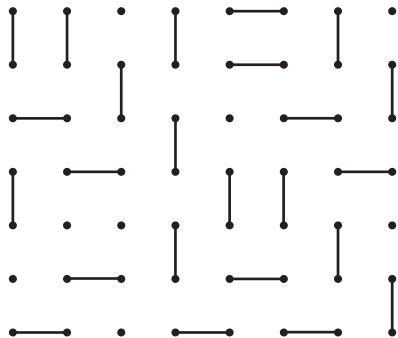}
\caption{A monomer--dimer covering in $\mathbb{Z}_{8 \times 7}$}
\label{fig:MD}
\end{figure}

The state matrix recursion method is divided into three stages; 
\begin{itemize}
\item Stage 1. Conversion to the mosaic system
\item Stage 2. State matrix recursion formula
\item Stage 3. Analyzing the state matrix
\end{itemize}
In Sections~\ref{sec:stage1}~$\! \sim \!$~\ref{sec:stage3},
we formulate the method and show the main result at the end.
Section~\ref{sec:growth} is devoted to the study of the asymptotic behavior of the growth rate of 
the Hosoya index of the $m \! \times \! n$ square lattice.
In Section~\ref{sec:fixed}, the dimer covering problem with multiple vacant sites is handled.

As an application of this method,
we also consider the domino tiling problem of the  Aztec diamond and its variant regions.
The Aztec diamond theorem from the excellent article of Elkies, Kuperberg, Larsen and Propp~\cite{EKLP}
states that the Aztec diamond of order $n$ can be tiled by dominos in exactly $2^{n(n+1)/2}$ ways.
A simple proof of this theorem can be found in~\cite{EF}.
An augmented Aztec diamond of order $n$ looks much like the Aztec diamond of order $n$,
except that there are three long columns in the middle instead of two.
Compare left two regions in Figure~\ref{fig:Aztec}.
The number of domino tilings of the augmented Aztec diamond of order $n$
was found by Sachs and Zernitz~\cite{SZ}
as $\sum_{k=0}^n {{n}\choose{k}} \cdot {{n+k}\choose{k}}$, known as the Delannoy numbers.
Notice that the former number is much larger than the later.
The enumeration problem of domino tilings of a region is known to be very sensitive 
to its boundary condition~\cite{MS1, MS2}.
Dozens of interesting patterns related to the Aztec diamond allowing some squares removed
have been deeply studied and a survey of these works was proposed by Propp~\cite{Pr}.
For example, see the rightmost figure showing a domino tiling of 
a 4-by-5 Aztec rectangle with its central square removed.

\begin{figure}[h]
\includegraphics{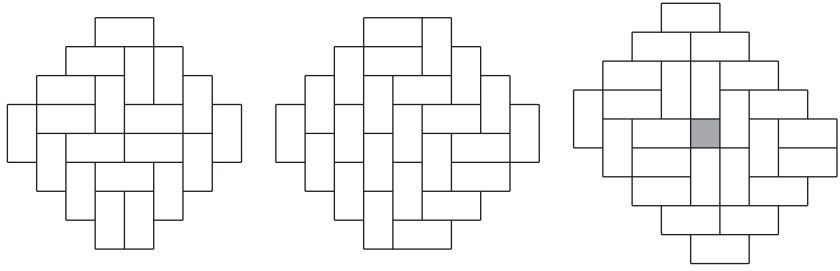}
\caption{Domino tilings of various Aztec regions}
\label{fig:Aztec}
\end{figure}

In Section~\ref{sec:domino}, 
we study the domino tilings of the most generalized region among Aztec diamond variants, 
called an Aztec octagon, obtained from the rectangular grid 
with four triangular corners (not necessary to be congruent) removed 
as drawn in Figure~\ref{fig:AztecOctagon}.

As another interesting application, 
this method provides a recursive matrix-relation producing the exact number of
independent vertex sets on the square lattice in papers~\cite{OhV1, OhV2}.

\section{Monomer--dimer problems} \label{sec:problem}

Let $\mathbb{Z}_{m \times n}$ denote the $m \! \times \! n$ rectangular grid on the square lattice.
A {\em dimer\/} is an edge connecting two nearest vertices.
Horizontal and vertical dimers are considered as $x$-dimers and $y$-dimers, respectively.
Dimers must be placed so that no vertex belongs to more than one dimer.
An unoccupied vertex is called a {\em monomer\/}.
The partition function of $\mathbb{Z}_{m \times n}$ with monomer and dimer activities,
assigned weights $v, x, y$ to monomers, $x$-dimers and $y$-dimers respectively, is defined by
$$G_{m \times n}(v,x,y) = \sum v^{n_v} x^{n_x} y^{n_y}$$
with respect to the numbers $n_v, n_x, n_y$ of monomers, $x$-dimers and $y$-dimers respectively, 
where the summation is taken over all monomer--dimer coverings.
Note that $n_v + 2(n_x + n_y) = mn$ in each term.

Based upon the state matrix recursion method,
we present a recursive matrix-relation producing this partition function.
Hereafter $\mathbb{O}_k$ denotes the $2^k \! \times \! 2^k$ zero-matrix.

\begin{theorem} \label{thm:dimer}
The partition function is
$$ G_{m \times n}(v,x,y) = (1,1)\mbox{-entry of } (A_m)^n, $$
where $A_m$ is the $2^m \! \times \! 2^m$ matrix\footnote{
In this theorem we may replace the recursive relation by
$$ A_{k} = A_{k-1 \otimes} \begin{bmatrix} v & 1 \\ y & 0 \end{bmatrix} +
A_{k-2} \otimes {\Small \begin{bmatrix} x & 0 & 0 & 0 \\ 0 & 0 & 0 & 0 \\ 
0 & 0 & 0 & 0 \\ 0 & 0 & 0 & 0 \end{bmatrix}}$$
in tensor product form.
This will be explained after the proof of Lemma~\ref{lem:bar}.} 
recursively defined by
$$ A_{k} = \begin{bmatrix} v A_{k-1} +
x {\Small \begin{bmatrix} A_{k-2} & \mathbb{O}_{k-2} \\ \mathbb{O}_{k-2} & \mathbb{O}_{k-2} \end{bmatrix}}
& A_{k-1} \\ y A_{k-1} & \mathbb{O}_{k-1} \end{bmatrix} $$
for $k=2, \dots, m$, with seed matrices
$A_0 = \begin{bmatrix} 1 \end{bmatrix}$ and
$A_1 = \begin{bmatrix} v & 1 \\ y & 0 \end{bmatrix}$.
\end{theorem}

Theorem~\ref{thm:dimer} presents a number of important consequences as follows.
First, we can derive the matching polynomial for $\mathbb{Z}_{m \times n}$
$$m_{\mathbb{Z}_{m \times n}}(z) = G_{m \times n}(1,z,z)$$
each of whose coefficient of $z^k$ indicates the number of $k$-edge matchings.

Second, $G_{m \times n}(1,1,1)$ gives the number of monomer--dimer coverings,
known as the Hosoya index of $\mathbb{Z}_{m \times n}$.
The Hosoya index~\cite{Ho} and the Merrifield-Simmons index~\cite{MS1, MS2} of a graph
are two prominent examples of topological indices 
which are used in mathematical chemistry for quantifying molecular-graph based structure descriptors.
The sequence of $G_{n \times n}(1,1,1)$, for $m \! = \! n$, grows in a quadratic exponential rate.
We focus on the asymptotic behavior of the growth rate per vertex.
Let
$$ \delta = \lim_{m, n \rightarrow \infty} (G_{m \times n}(1,1,1))^{\frac{1}{mn}}, $$
provided that it exists.
The existence of that limit was proved in~\cite{HLLB}.
A two-dimensional application of the Fekete's lemma shows again the existence of the limit.
The following theorem will be proved in Section~\ref{sec:growth}.

\begin{theorem} \label{thm:growth}
The double limit $\delta$ exists.
More precisely,
$$ \delta = \sup_{m, n \geq 1} (G_{m \times n}(1,1,1))^{\frac{1}{mn}}.$$
\end{theorem}

Third, known as the pure dimer problem for even $mn$,
$G_{m \times n}(0,x,y)$ is the partition function of $\mathbb{Z}_{m \times n}$ only with dimer activity,
assigned weights $x, y$ to $x$-dimers and $y$-dimers respectively.
Remark that, instead of the form $G_{m \times n}(0,1,1)$ of the number of pure dimer coverings,
a better closed form of this number was already founded as mentioned in the introduction.
Hammersley~\cite{Ha} showed that the following limit exists and
from the exact results~\cite{Ka1, TF}, we know that
$$ \lim_{n \rightarrow \infty} (G_{2n \times 2n}(0,1,1))^{\frac{1}{4n^2}} = e^{\frac{C}{\pi}} = 1.338515\cdots,$$
where $C$ is the Catalan's constant.

Fourth, the coefficient of the degree 1 term $v$ of $G_{m \times n}(v,1,1)$ indicates
the number of dimer coverings with a single vacancy (non-fixed and on/off the boundary) for odd $m$ and $n$.
But, more interesting models are dimer coverings 
with a fixed single vacancy on the boundary \cite{Ko, TW, Wu}.
A fixed single boundary monomer covering, say, in $\mathbb{Z}_{m \times n}$ for odd $m$ and $n$
is a monomer--dimer covering with exactly one fixed monomer on the boundary,
having odd-numbered $x$- and $y$-coordinates.
It is known that the number of dimer coverings with fixed single boundary monomer does not depend on
the location of the fixed monomer~\cite{TW}.

\begin{theorem} \label{thm:single}
Let $G^s_{m \times n}(v,x,y)$ be the $(2,1)$-entry of $(A_m)^n$ in Theorem~\ref{thm:dimer}
for odd $m$ and $n$.
Then $G^s_{m \times n}(0,1,1)$ is the number of fixed single boundary monomer coverings
in $\mathbb{Z}_{m \times n}$.
\end{theorem}

Note that, instead of the (2,1)-entry,
we may use any $(i,j)$-entry of $(A_m)^n$ for $\{i,j\} = \{1,2^k \! + \! 1\}$ and $k=0,2,4,\dots, m \! - \! 1$.

We are turning now to a generalization of this fixed monomer argument so that
many sites are pre-assigned to monomers.
Let $S$ be a set of vertices in $\mathbb{Z}_{m \times n}$,
called a fixed monomer set as in Figure~\ref{fig:fixedconf}.
In this case, we only consider the number of monomer--dimer coverings
instead of the partition function with monomer and dimer activities,
by assigning 1 to the weights $v$, $x$ and $y$.
$g_{m \times n}(S)$ denotes the number of distinct monomer--dimer coverings
which have monomers exactly at the sites of $S$.
Here $(k,i)$ indicates the vertex placed at the $k$th column from left to right and
the $i$th row from bottom to top.

\begin{theorem} \label{thm:fixed}
For a given fixed monomer set $S$ in $\mathbb{Z}_{m \times n}$,
$$ g_{m \times n}(S) = (1,1)\mbox{-entry of } \prod^n_{i=1} A_{m,i}, $$
where $A_{m,i}$ is defined by the recurrence relations, for $k=1, \dots, m$, \\
if the vertex $(k,i)$ is contained in $S$,
$$ A_{k,i} = \begin{bmatrix} A_{k-1,i} & \mathbb{O}_{k-1} \\
\mathbb{O}_{k-1} & \mathbb{O}_{k-1} \end{bmatrix} \mbox{ and }
 B_{k,i} = \mathbb{O}_{k} $$
or if the vertex $(k,i)$ is not contained in $S$,
$$ A_{k,i} = \begin{bmatrix} B_{k-1,i} & A_{k-1,i} \\
A_{k-1,i} & \mathbb{O}_{k-1} \end{bmatrix} \mbox{ and }
 B_{k,i} = \begin{bmatrix} A_{k-1,i} & \mathbb{O}_{k-1} \\
\mathbb{O}_{k-1} & \mathbb{O}_{k-1} \end{bmatrix} $$
with seed matrices
$A_{0,i} = \begin{bmatrix} 1 \end{bmatrix}$ and $B_{0,i}= \begin{bmatrix} 0 \end{bmatrix}$.
\end{theorem}

\begin{figure}[h]
\includegraphics{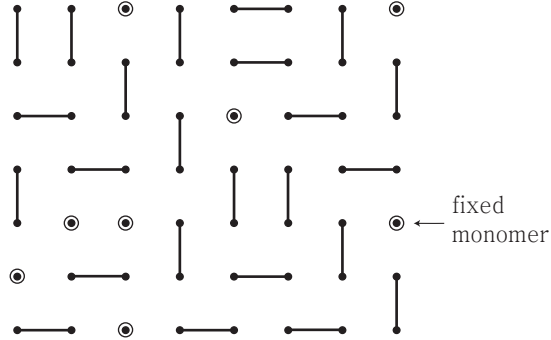}
\caption{A monomer--dimer covering with fixed monomers}
\label{fig:fixedconf}
\end{figure}

\section{Stage 1. Conversion to the monomer--dimer mosaic system} \label{sec:stage1}

This stage is dedicated to the installation of 
the mosaic system for monomer--dimer coverings on the square lattice.
Mosaic system is introduced by Lomonaco and Kauffman \cite{LK} to give
a precise and workable definition of quantum knots.
This definition is intended to represent an actual physical quantum system.

Recently, the author {\em et al\/}. have developed a state matrix argument 
for knot mosaic enumeration~\cite{HLLO2, HO, Oh1, OHLL, OHLLY}.
We follow the notation and terminology used in~\cite{OHLL} 
with much modification to adjust to the dimer system.

Five symbols $T_1$, $T_2$, $T_3$, $T_4$ and $T_5$ 
illustrated in Figure~\ref{fig:tile} are called {\em mosaic tiles\/}
(for monomer--dimer coverings on the square lattice).
Their side edges are labeled with two letters {\texttt a} and {\texttt b} as follows:
letter {\texttt a} if it is not touched by a thick arc on the tile, and letter {\texttt b} for otherwise.
In the original definition of knot mosaic theory, eleven symbols were used to represent a knot diagram.

\begin{figure}[h]
\includegraphics{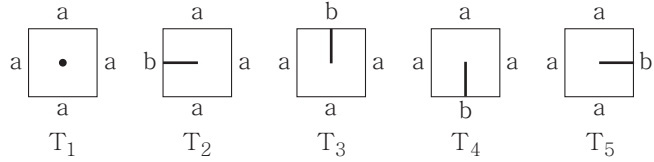}
\caption{Five mosaic tiles labeled with two letters}
\label{fig:tile}
\end{figure}

For positive integers $m$ and $n$,
an {\em $m \! \times \! n$-mosaic\/} is an $m \! \times \! n$ rectangular array $M = (M_{ij})$ of those tiles,
where $M_{ij}$ denotes the mosaic tile placed at the $i$th column from left to right
and the $j$th row from bottom to top. 
We are exclusively interested in mosaics whose tiles match each other properly
to represent monomer--dimer coverings.
This requires the followings: \vspace{2mm}

\begin{itemize}[leftmargin=*] \itemsep5pt
\item (Adjacency rule)
Abutting edges of adjacent mosaic tiles in a mosaic are labeled with the same letter.
\item (Boundary state requirement)
All boundary edges in a mosaic are labeled with letter {\texttt a}.
\end{itemize}
\vspace{2mm}

As illustrated in Figure~\ref{fig:conversion},
every monomer--dimer covering in $\mathbb{Z}_{m \times n}$ can be converted into 
an $m \! \times \! n$-mosaic which satisfies the two rules.
In this mosaic, a dot in each $T_1$ indicates a monomer,
and $T_2$ and $T_5$ (or, $T_3$ and $T_4$) can be adjoined 
along the edges labeled {\texttt b} to produce a dimer.
Note that the statements of the adjacency rule and boundary state requirement vary in different lattice models.

\begin{figure}[h]
\includegraphics{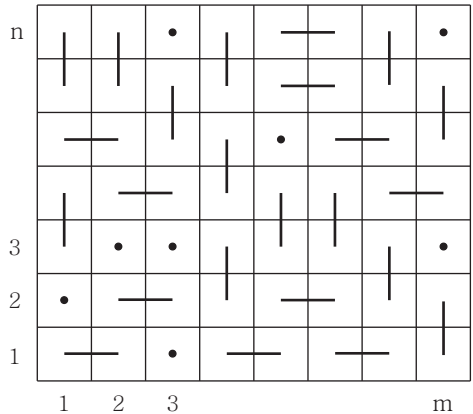}
\caption{Conversion of the monomer--dimer covering drawn in Figure~\ref{fig:MD} 
to a monomer--dimer $m \! \times \! n$-mosaic}
\label{fig:conversion}
\end{figure}

A mosaic is said to be {\em suitably adjacent\/} if any pair of mosaic tiles
sharing an edge satisfies the adjacency rule.
A suitably adjacent $m \! \times \! n$-mosaic is called a {\em monomer--dimer $m \! \times \! n$-mosaic\/}
if it additionally satisfies the boundary state requirement.
Key role is the following one-to-one conversion. 

\begin{one-to-one conversion} \label{observ}
There is a one-to-one correspondence between  
monomer--dimer coverings in $\mathbb{Z}_{m \times n}$ and monomer--dimer $m \! \times \! n$-mosaics.
\end{one-to-one conversion}

\section{Stage 2. State matrix recursion formula} \label{sec:stage2}

Now we introduce two types of state matrices for suitably adjacent mosaics
to produce the partition function $G_{m \times n}(v,x,y)$.

\subsection{States and state polynomials}

Let $p \leq m$ and $q \leq n$ be positive integers, 
and consider a suitably adjacent $p \! \times \! q$-mosaic $M$.
A {\em state\/} is a finite sequence of two letters {\texttt a} and {\texttt b}.
The {\em $b$-state\/} $s_b(M)$ ({\em $t$-state\/} $s_t(M)$) is 
the state of length $p$ obtained by reading off letters on the bottom (top, respectively) 
boundary edges of $M$ from right to left, 
and the {\em $l$-state\/} $s_l(M)$ ({\em $r$-state\/} $s_r(M)$) is 
the state of length $q$ on the left (right, respectively) 
boundary edges from top to bottom as shown in Figure~\ref{fig:arrow}.
State {\texttt a}{\texttt a}$\cdots${\texttt a} is called trivial.

\begin{figure}[h]
\includegraphics{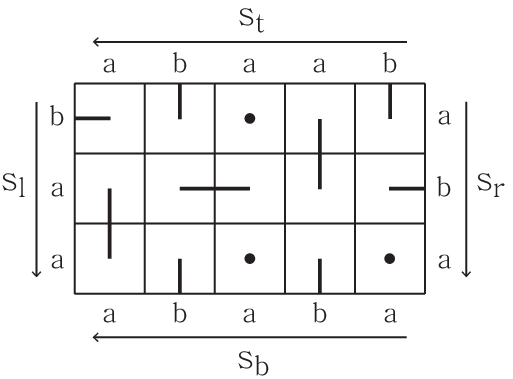}
\caption{A suitably adjacent $5 \! \times \! 3$-mosaic with four state indications:
$s_r(M) =$ \texttt{aba}, $s_b(M) =$ \texttt{ababa}, $s_t(M) =$ \texttt{baaba}, and $s_l(M) =$ \texttt{baa}}
\label{fig:arrow}
\end{figure}
 
Given a triple $\langle s_r, s_b, s_t \rangle$ of $r$-, $b$- and $t$-states,
we associate the {\em state polynomial\/}:
$$ P_{\langle s_r, s_b, s_t  \rangle}(v,x,y) = \sum k(n_v, n_x, n_y) \, v^{n_v} x^{n_x} y^{n_y}, $$
where $k(n_v, n_x, n_y)$ equals the number of all suitably adjacent $p \! \times \! q$-mosaics $M$,
having $n_v$, $n_x$, $n_y$ numbers of $T_1$, $T_2$, $T_4$ mosaic tiles, respectively,
such that $s_r(M) = s_r$, $s_b(M) = s_b$, $s_t(M) = s_t$ and trivial $s_l(M)=$ {\texttt a}{\texttt a}$\cdots${\texttt a}.
Mosaic tiles $T_1$, $T_2$, $T_4$ are respectively related to 
a monomer, an $x$-dimer's right part and a $y$-dimer's top part.
The last triviality condition of $s_l(M)$ is necessary for the left boundary state requirement.
See Figure~\ref{fig:polyexam} for an explicit example.

\begin{figure}[h]
\includegraphics{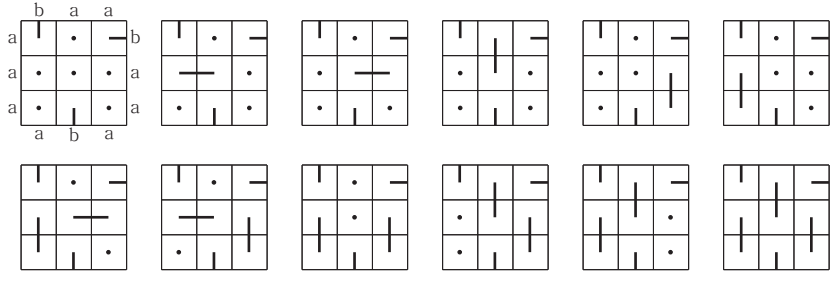}
\caption{Twelve suitably adjacent $3 \! \times \! 3$-mosaics producing
$P_{\langle \texttt{baa}, \texttt{aba}, \texttt{aab} \rangle}(v,x,y) = 
v^6 y + 2v^4 xy + 3v^4 y^2 + 2v^2 x y^2 + 3v^2 y^3 + y^4$}
\label{fig:polyexam}
\end{figure}

\subsection{Bar state matrices}

Consider suitably adjacent $p \! \times \! 1$-mosaics, which are called {\em bar mosaics\/}.
Bar mosaics of length $p$ have possibly $2^p$ kinds of $b$- and $t$-states, especially called {\em bar states\/}.
We arrange all bar states in the lexicographic order.
For $1 \leq i \leq 2^p$, let $\epsilon^p_i$ denote the $i$th bar state in this order.

{\em Bar state matrix\/} $X_p$ ($X = A, B$)
for the set of suitably adjacent bar mosaics of length $p$ is a $2^p \! \times \! 2^p$ matrix $(m_{ij})$ given by  
$$ m_{ij} = P_{\langle \text{x}, \epsilon^p_i, \epsilon^p_j \rangle}(v,x,y), $$
where x $=$ {\texttt a}, {\texttt b}, respectively.
We remark that information on suitably adjacent bar mosaics with trivial $l$-state is completely encoded 
in two bar state matrices $A_p$ and $B_p$.

\begin{lemma} \label{lem:bar}
Bar state matrices $A_p$ and $B_p$ are obtained by the recurrence relations:
$$ A_{k} = \begin{bmatrix} v A_{k-1} + x B_{k-1} & A_{k-1} \\ y A_{k-1} & \mathbb{O}_{k-1} \end{bmatrix}
\mbox{ and }
B_{k} = \begin{bmatrix} A_{k-1} & \mathbb{O}_{k-1} \\ \mathbb{O}_{k-1} & \mathbb{O}_{k-1} \end{bmatrix} $$
with seed matrices
$A_1 = \begin{bmatrix} v & 1 \\ y & 0 \end{bmatrix}$ and $B_1 = \begin{bmatrix} 1 & 0 \\ 0 & 0 \end{bmatrix}$.
\end{lemma}

Note that we may start with matrices
$A_0 = \begin{bmatrix} 1 \end{bmatrix}$ and $B_0 = \begin{bmatrix} 0 \end{bmatrix}$
instead of $A_1$ and $B_1$.

\begin{proof}
We use induction on $k$.
A straightforward observation on four mosaic tiles $T_1$, $T_3$, $T_4$ and $T_5$ 
establishes the lemma for $k=1$.
For example, $(2,1)$-entry of $A_1$ is 
$$ P_{\langle {\texttt a}, \epsilon^1_2, \epsilon^1_1 \rangle}(v,x,y) =
P_{\langle {\texttt a}, {\texttt b}, {\texttt a} \rangle}(v,x,y) = y $$
since only mosaic tile $T_4$ satisfies this requirement.

Assume that $A_{k-1}$ and $B_{k-1}$ satisfy the statement.
For one case, we consider $A_{k}$.
Partition this matrix of size $2^k \! \times \! 2^k$
into four block submatrices of size $2^{(k-1)} \! \times \! 2^{(k-1)}$, 
and consider the 11-submatrix of $A_{k}$,
i.e., the $(1,1)$-component in the $2 \! \times \! 2$ array of the four blocks.
The $(i,j)$-entry of this 11-submatrix is the state polynomial 
$P_{\langle {\texttt a}, {\texttt a}\epsilon^{k-1}_i, {\texttt a}\epsilon^{k-1}_j \rangle}(v,x,y)$
where {\texttt a}$\epsilon^{k-1}_i$ (similarly {\texttt a}$\epsilon^{k-1}_j$) is a bar state of length $k$
obtained by concatenating two states {\texttt a} and $\epsilon^{k-1}_i$.
A suitably adjacent $k \! \times \! 1$-mosaic corresponding to 
this triple $\langle {\texttt a}, {\texttt a}\epsilon^{k-1}_i, {\texttt a}\epsilon^{k-1}_j \rangle$
has two choices $T_1$ and $T_2$ for the rightmost mosaic tile,
and so its second rightmost tile must have $r$-state {\texttt a} or {\texttt b}, respectively, by the adjacency rule.
By considering the contribution of the rightmost tiles $T_1$ and $T_2$ to the state polynomial,
one easily gets
$$P_{\langle {\texttt a}, {\texttt a}\epsilon^{k-1}_i, {\texttt a}\epsilon^{k-1}_j \rangle}(v,x,y) =
v \big( (i,j)\text{-entry of } A_{k-1} \big) + x \big( (i,j)\text{-entry of } B_{k-1}\big).$$
Thus the 11-submatrix of $A_{k}$ is $v A_{k-1} + x B_{k-1}$.
See Figure~\ref{fig:barset}.

\begin{figure}[h]
\includegraphics{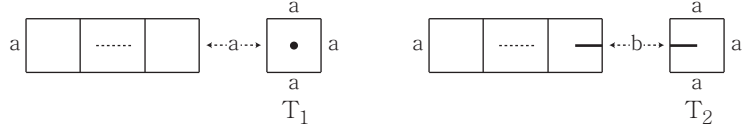}
\caption{Expanding a bar mosaic}
\label{fig:barset}
\end{figure}

All the other cases have no or unique choice for the rightmost mosaic tile and 
the same argument gives Table~\ref{tab:barset} presenting all possible eight cases as desired.
\end{proof}

\begin{table}[h]
\bgroup
\def\arraystretch{1.2} 
{\small
\begin{tabular}{cccc}      \hline \hline
  & \ {\em Submatrix for\/} $\langle s_r, s_b, s_t \rangle$ \ & \ {\em Rightmost tile\/} \ &
{\em Submatrix\/} \\    \hline
\multirow{4}{4mm}{$A_k$}
 & 11-submatrix $\langle {\texttt a}, {\texttt a} \! \cdot \! \cdot, {\texttt a} \! \cdot \! \cdot \rangle$ 
 & $T_1$, $T_2$ &  \ $v A_{k-1} + x B_{k-1}$ \ \\
 & 12-submatrix $\langle {\texttt a}, {\texttt a} \! \cdot \! \cdot, {\texttt b} \! \cdot \! \cdot \rangle$
 & $T_3$ & $A_{k-1}$ \\
 & 21-submatrix $\langle {\texttt a}, {\texttt b} \! \cdot \! \cdot, {\texttt a} \! \cdot \! \cdot \rangle$
 & $T_4$ & $y A_{k-1}$ \\
 & 22-submatrix $\langle {\texttt a}, {\texttt b} \! \cdot \! \cdot, {\texttt b} \! \cdot \! \cdot \rangle$
 & None & $\mathbb{O}_{k-1}$ \\    \hline
\multirow{4}{4mm}{$B_k$}
 & 11-submatrix $\langle {\texttt b}, {\texttt a} \! \cdot \! \cdot, {\texttt a} \! \cdot \! \cdot \rangle$
 & $T_5$ & $A_{k-1}$ \\
 & 12-submatrix $\langle {\texttt b}, {\texttt a} \! \cdot \! \cdot, {\texttt b} \! \cdot \! \cdot \rangle$
 & None & $\mathbb{O}_{k-1}$ \\
 & 21-submatrix $\langle {\texttt b}, {\texttt b} \! \cdot \! \cdot, {\texttt a} \! \cdot \! \cdot \rangle$
 & None & $\mathbb{O}_{k-1}$ \\
 & 22-submatrix $\langle {\texttt b}, {\texttt b} \! \cdot \! \cdot, {\texttt b} \! \cdot \! \cdot \rangle$
 & None & $\mathbb{O}_{k-1}$ \\  \hline \hline
\end{tabular}
}
\egroup
\vspace{4mm}
\caption{Eight submatrices of $A_k$ and $B_k$}
\label{tab:barset}
\end{table}

Remark that we may replace the recursive relation in Lemma~\ref{lem:bar} by
$$ A_{k} = A_{k-1} \otimes \begin{bmatrix} v & 1 \\ y & 0 \end{bmatrix} +
B_{k-1} \otimes \begin{bmatrix} x & 0 \\ 0 & 0 \end{bmatrix}
\mbox{ and }
B_{k} = A_{k-1} \otimes \begin{bmatrix} 1 & 0 \\ 0 & 0 \end{bmatrix} $$
in tensor product form.
This will be done by re-defining $b$- and $t$-states so as reading off $m$-tuple of states
on the bottom and top, respectively, boundary edges from left to right (the reverse direction).
Now follow the same argument as in the above proof.

\subsection{State matrices}

{\em State matrix\/} $A_{m \times q}$ for the set of suitably adjacent $m \! \times \! q$-mosaics 
is a $2^m \! \times \! 2^m$ matrix $(a_{ij})$ given by 
$$ a_{ij} = P_{\langle {\texttt a} \cdots {\texttt a}, \epsilon^m_i, \epsilon^m_j \rangle}(v,x,y). $$
The trivial state condition of $s_r$ is necessary for the right boundary state requirement.
We get state matrix $A_{m \times n}$ by simply multiplying the bar state matrix $n$ times.

\begin{lemma} \label{lem:mn}
State matrix $A_{m \times n}$ is obtained by
$$ A_{m \times n} = (A_m)^n. $$
\end{lemma}

\begin{proof}
We use induction on $n$.
For $n=1$, $A_{m \times 1} = A_m$
since $A_{m \times 1}$ counts suitably adjacent $m \! \times \! 1$-mosaics with trivial $r$-state {\texttt a}.

Assume that $A_{m \times (k-1)} = (A_m)^{k-1}$.
Let $M^{m \times k}$ be a suitably adjacent $m \! \times \! k$-mosaic with trivial $l$- and $r$-states.
Also let $M^{m \times (k-1)}$ and $M^{m \times 1}$ be the suitably adjacent 
$m \! \times \! (k \! - \! 1)$- and $m \! \times \! 1$-mosaics
by splitting bottom $k \! - \! 1$ bar mosaics and the top bar mosaic.
By the adjacency rule,
the $t$-state of $M^{m \times (k-1)}$ and the $b$-state of $M^{m \times 1}$
must coincide as shown in Figure~\ref{fig:expand}.

\begin{figure}[h]
\includegraphics{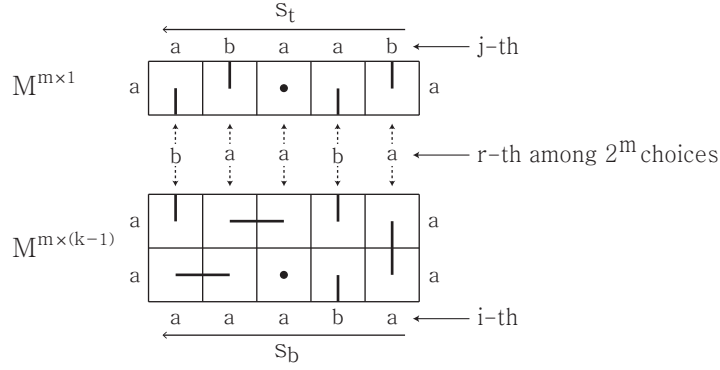}
\caption{Expanding $M^{m \times (k-1)}$ to $M^{m \times k}$}
\label{fig:expand}
\end{figure}

Let $A_{m \times k} = (a_{ij})$, $A_{m \times (k-1)} = (a_{ij}')$ and $A_{m \times 1} = (a_{ij}'')$.
Note that $a_{ij}$ is the state polynomial for the set of suitably adjacent $m \! \times \! k$-mosaics
$M$ which admit splittings into $M^{m \times (k-1)}$ and $M^{m \times 1}$ satisfying
$s_b(M) = s_b(M^{m \times (k-1)}) = \epsilon^m_i$,
$s_t(M) = s_t(M^{m \times 1}) = \epsilon^m_j$, and
$s_t(M^{m \times (k-1)}) = s_b(M^{m \times 1}) = \epsilon^m_r$ ($1 \leq r \leq 2^m$).
Obviously, all $l$- and $r$-states of them must be trivial.
Thus, 
$$ a_{ij} = \sum^{2^m}_{r=1} a_{ir}'  \cdot  a_{rj}''. $$
This implies
$$ A_{m \times k} = A_{m \times (k-1)} \cdot A_{m \times 1} = (A_m)^k, $$
and the induction step is finished
\end{proof}

\section{Stage 3. Analyzing the state matrix} \label{sec:stage3}

We analyze state matrix $A_{m \times n} = (A_m)^n$ to find the partition function $G_{m \times n}(v,x,y)$.
  
\begin{proof}[Proof of Theorem~\ref{thm:dimer}.]
The $(1,1)$-entry of $A_{m \times n}$ is the state polynomial for the set of 
suitably adjacent $m \! \times \! n$-mosaics associated to the triple 
$$ \langle {\texttt a} \cdots {\texttt a}, \epsilon^m_1, \epsilon^m_1 \rangle =
\langle {\texttt a} \cdots {\texttt a}, {\texttt a} \cdots {\texttt a}, {\texttt a} \cdots {\texttt a} \rangle, $$
so having trivial $r$-, $b$-, $t$- and $l$-states.
According to the boundary state requirement,
monomer--dimer coverings in $\mathbb{Z}_{m \times n}$ are converted into 
suitably adjacent $m \! \times \! n$-mosaics $M$ with trivial $r$-, $b$-, $t$- and $l$-states
as the left picture in Figure~\ref{fig:analyze}.
Thus this state polynomial represents the partition function $G_{m \times n}(v,x,y)$.
In short, we get 
$$G_{m \times n}(v,x,y) = \mbox{(1,1)-entry of } A_{m \times n}.$$
This combined with Lemmas~\ref{lem:bar} and \ref{lem:mn} completes the proof.
Note that the two recurrence relations in Lemma~\ref{lem:bar}
easily merge into one recurrence relation as in Theorem~\ref{thm:dimer}.
\end{proof}

\begin{figure}[h]
\includegraphics{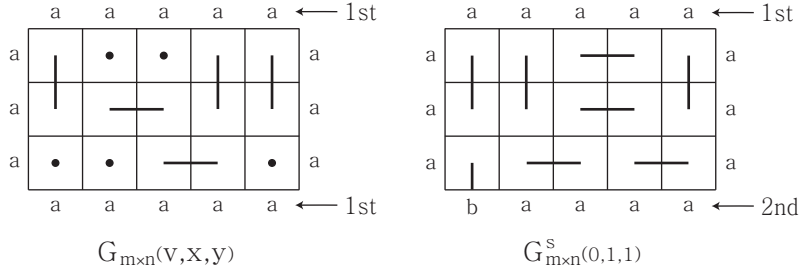}
\caption{Examples of monomer--dimer coverings related to $(1,1)$- and $(2,1)$-entries of $A_{m \times n}$}
\label{fig:analyze}
\end{figure}

\begin{proof}[Proof of Theorem~\ref{thm:single}.]
$G^s_{m \times n}(v,x,y)$ which is the $(2,1)$-entry of $(A_m)^n$ is the state polynomial 
associated  to the triple $\langle {\texttt a} \cdots {\texttt a}, \epsilon^m_2, \epsilon^m_1 \rangle$,
so having the second $b$-state and trivial $t$-, $l$- and $r$-states.
Since the second $b$-state is {\texttt a}{\texttt a}$\cdots${\texttt a}{\texttt b},
$M_{1,1}$ must be mosaic tile $T_4$ and we may consider it as a fixed single monomer.
Now $G^s_{m \times n}(0,1,1)$ is the number of pure dimer coverings
with a single monomer at $M_{1,1}$, as desired.
It is well-known that this number is independent of location of the monomer,
provided that it places at boundary sites with odd-numbered $x$- and $y$-coordinates~\cite{TW}.
Therefore, instead of the (2,1)-entry,
we may use any $(i,j)$-entry of $(A_m)^n$ for $\{i,j\} = \{1,2^k \! + \! 1\}$ and $k=0,2,4,\dots, m \! - \! 1$.
\end{proof}

\section{Growth constant of the Hosoya index} \label{sec:growth}

We will need the following result called Fekete's lemma 
whose consequences are many and deep.
In this paper we state and prove its two-variate multiplicative version with generalization.

\begin{lemma}[Generalized Fekete's Lemma] \label{lem:Fekete}
Let $\{ a_{m,n} \}_{m, \, n \in \, \mathbb{N}}$ be a double sequence with $a_{m,n} \geq 1$,
and $k$ be a nonnegative integer.

If the sequence satisfies
$a_{m_1,n} \cdot a_{m_2,n} \leq a_{m_1 + m_2 + k,n}$
and $a_{m,n_1} \cdot a_{m,n_2} \leq a_{m,n_1 + n_2 + k}$
for all $m$, $m_1$, $m_2$, $n$, $n_1$ and $n_2$,
then
$$ \lim_{m, n \rightarrow \infty} (a_{m,n})^{\frac{1}{mn}} = 
\sup_{m, n \geq 1} (a_{m,n})^{\frac{1}{(m+k)(n+k)}}, $$
provided that the supremum exists.

Instead, if it satisfies
$a_{m_1+m_2,n} \leq a_{m_1+k,n} \cdot a_{m_2,n}$ 
and $a_{m,n_1+n_2} \leq a_{m,n_1+k} \cdot a_{m,n_2}$,
then
$$ \lim_{m, n \rightarrow \infty} (a_{m,n})^{\frac{1}{mn}} = 
\inf_{m, n > k} (a_{m,n})^{\frac{1}{(m-k)(n-k)}}. $$
\end{lemma}

Remark that in this paper we only use the supermultiplicative inequality part with $k=0$.
The other parts will be used in on-going papers.

\begin{proof}
Let $S = \sup_{m, n} (a_{m,n})^{\frac{1}{(m+k)(n+k)}}$ and let $B$ be any number less than~$S$.
Choose any positive integers $i$ and $j$ satisfying $ B < (a_{i,j})^{\frac{1}{(i+k)(j+k)}}$.
For sufficiently large integers $m$ and $n$, 
there are integers $p_m$ and $q_m$ (simillary $p_n$ and $q_n$ for $n$ and $j$)
such that $m = p_m (i \! + \! k) + q_m$ and $0 \leq q_m < i \! + \! k$ by the division algorithm.
By the supermultiplicative inequalities prescribed in the lemma,
$$ (a_{m,n})^{\frac{1}{mn}} \geq (a_{i,n})^{\frac{p_m}{mn}} \geq (a_{i,j})^{\frac{p_m p_n}{mn}} = 
(a_{i,j})^{\frac{1}{(i+k)(j+k)} \Big(\frac{p_m (i+k)}{m}\Big) \Big(\frac{p_n (j+k)}{n}\Big)}. $$
Since $\frac{p_m (i+k)}{m}, \frac{p_n (j+k)}{n} \rightarrow 1$ as $m, n \rightarrow \infty$,
we have
$$ B < (a_{i,j})^{\frac{1}{(i+k)(j+k)}} \leq \lim_{m, \, n \rightarrow \infty} (a_{m,n})^{\frac{1}{mn}} \leq S. $$
This provides the desired limit.
The submultiplicative inequality part of the proof can be proved in similar way.
\end{proof}

\begin{proof}[Proof of Theorem~\ref{thm:growth}.]
We use briefly $G_{m \times n}$ to denote $G_{m \times n}(1,1,1)$
which is obviously at least 1 for all $m$, $n$.
First, we prove the existence of the limit of $(G_{m \times n})^{\frac{1}{mn}}$.
The supermultiplicative inequalities
$G_{m_1 \times n} \cdot G_{m_2 \times n} \leq G_{(m_1+m_2) \times n}$ and similarly
$G_{m \times n_1} \cdot G_{m \times n_2} \leq G_{m \times (n_1+n_2)}$
are obvious because we can create a new monomer--dimer $(m_1 \! + \! m_2) \! \times \! n$-mosaic
by simply adjoining two monomer--dimer $m_1 \! \times \! n$- and $m_2 \! \times \! n$-mosaics.
Since $\sup_{m, \, n} (G_{m \times n})^{\frac{1}{mn}} \leq 5$
which is the number of possible mosaic tiles at each site,
we apply Lemma~\ref{lem:Fekete}.
\end{proof}

\section{Fixed monomers problem} \label{sec:fixed}

\begin{proof}[Proof of Theorem~\ref{thm:fixed}.]
Let $S$ be a fixed monomer set.
We find bar state matrix $A_{m,i}$ for some $i$th bar mosaic with fixed monomers by using
relevant bar state matrix recurrence relations similar to Lemma~\ref{lem:bar}
with some modifications in each step $(k,i)$ as below.
We may assume that $\{(p,i), \dots, (q,i) \}$, $1 \leq p \leq q \leq m$, is a subset of maximal consecutive vertices
in $S$ on the $i$th bar mosaic in the sense that two vertices $(p \! - \! 1,i)$ and $(q \! + \! 1,i)$
(if they exist) are not contained in $S$.
See Figure~\ref{fig:fixed2}.
Note that, in this case, we only consider the number of monomer--dimer coverings instead of the partition function.

\begin{figure}[h]
\includegraphics{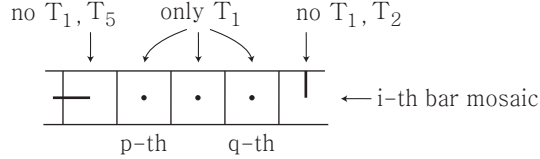}
\caption{A set of maximal consecutive vertices of $S$}
\label{fig:fixed2}
\end{figure}

If $(k \! - \! 1,i)$, $(k,i)$ and $(k \! + \! 1,i)$ are not contained in $S$,
\begin{equation}
A_{k,i} = \begin{bmatrix} B_{k-1,i} & A_{k-1,i} \\
A_{k-1,i} & \mathbb{O}_{k-1} \end{bmatrix} \mbox{ and }
B_{k,i} = \begin{bmatrix} A_{k-1,i} & \mathbb{O}_{k-1} \\
\mathbb{O}_{k-1} & \mathbb{O}_{k-1} \end{bmatrix}    
\end{equation}
because $T_1$ cannot be used in this step.
Also if $k=p, \dots, q$, i.e., $(k,i)$ is contained in $S$,
$$ A_{k,i} = \begin{bmatrix} A_{k-1,i} & \mathbb{O}_{k-1} \\
\mathbb{O}_{k-1} & \mathbb{O}_{k-1} \end{bmatrix} \mbox{ and }
B_{k,i} = \mathbb{O}_{k} $$
because only $T_1$ can be located at $(k,i)$.

In the remaining cases of $k=p \! - \! 1$ or $q \! + \! 1$,
$$ A_{p-1,i} = \begin{bmatrix} B_{p-2,i} & A_{p-2,i} \\
A_{p-2,i} & \mathbb{O}_{p-2} \end{bmatrix} \mbox{ and }
B_{p-1,i} = \mathbb{O}_{p-1} $$
because $T_1$ and $T_5$ cannot be located at $(p \! - \! 1,i)$
and so the 11-submatrix of $B_{p-1,i}$ is $\mathbb{O}_{p-2}$ instead of $A_{p-2,i}$,
and 
$$ A_{q+1,i} = \begin{bmatrix} \mathbb{O}_{q} & A_{q,i} \\
A_{q,i} & \mathbb{O}_{q} \end{bmatrix} \mbox{ and }
B_{q+1,i} = \begin{bmatrix} A_{q,i} & \mathbb{O}_{q} \\
\mathbb{O}_{q} & \mathbb{O}_{q} \end{bmatrix} $$
because $T_1$ and $T_2$ cannot be located at $(q \! + \! 1,i)$.
Indeed, the equations in these remaining cases can be replaced by Eq.~(1)
because $A_{p,i}$ and $B_{p,i}$ do not use $B_{p-1,i}$ and
the 11-submatrix of $A_{q+1,i}$ is $\mathbb{O}_{q}$ which is equal to $B_{q,i}$.

Therefore all of these recurrence relations eventually merge into 
the recurrence relations in Theorem~\ref{thm:fixed}.
By applying the rest of the state matrix recursion method,
we conclude that
$$ g_{m \times n}(S) = \mbox{(1,1)-entry of } \prod^n_{i=1} A_{m,i}, $$
which completes the proof.
\end{proof}

\section{Domino tilings in the Aztec octagon} \label{sec:domino}

An Aztec diamond of order $n$ consists of all lattice squares 
that lie completely inside the diamond shaped region $\{ (x,y) : |x|+|y| \leq n+1 \}$. 
An augmented Aztec diamond of order $n$ looks much like the Aztec diamond of order $n$,
except that there are three long columns in the middle instead of two.
A {\em domino\/} is a 1-by-2 or 2-by-1 rectangle.
There are exact enumerations of domino tilings of these two regions 
and dozens of interesting variants as stated in the introduction.

In this section, we study the domino tilings on the most extended version of the Aztec diamond.
An $m \! \times \! n$-{\em Aztec octagon\/} of order $(p,q,r,s)$, denoted by $\mathbb{A}_{m \times n}(p,q,r,s)$,
is defined as the union of  
$mn -\frac{1}{2}(p^2 \! - \! p+ \! q^2 \! - \! q \! + \! r^2 \! - \! r \! + \! s^2 \! - \! s)$ unit squares,
arranged in the $m \! \times \! n$ rectangular grid with four triangular corners 
with side lengths $p \! - \! 1$, $q \! - \! 1$, $r \! - \! 1$, $s \! - \! 1$ removed in clockwise order, 
as drawn in Figure~\ref{fig:AztecOctagon}.
Aztec diamond and augmented Aztec diamond of order $n$
can be represented as $\mathbb{A}_{2n \times 2n}(n,n,n,n)$ 
and $\mathbb{A}_{(2n+1) \times 2n}(n,n,n,n)$, respectively, and
the $m \! \times \! n$ rectangular region $\mathbb{Z}_{m \times n}$
is indeed $\mathbb{A}_{m \times n}(1,1,1,1)$.

\begin{figure}[h]
\includegraphics{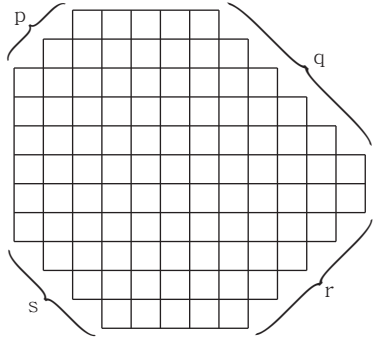}
\caption{$\mathbb{A}_{12 \times 11}(3,6,5,4)$}
\label{fig:AztecOctagon}
\end{figure}

\begin{figure}[h]
\includegraphics{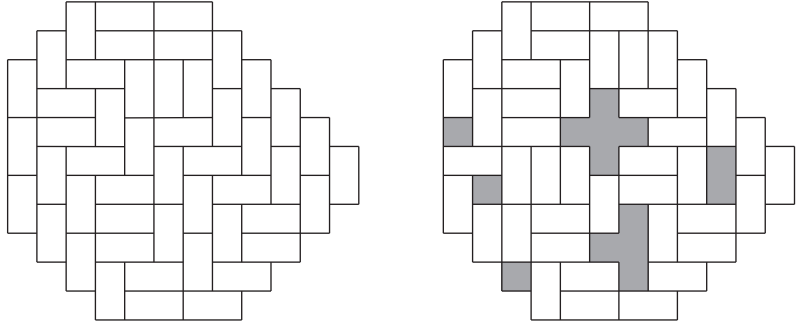}
\caption{Domino tilings of Aztec octagons without/with holes}
\label{fig:domino}
\end{figure}

Let $\alpha_{m \times n}(p,q,r,s)$ denote the number of domino tilings of $\mathbb{A}_{m \times n}(p,q,r,s)$.

\begin{theorem} \label{thm:aztec}
For domino tilings of an Aztec octagon $\mathbb{A}_{m \times n}(p,q,r,s)$, 
$$ \alpha_{m \times n}(p,q,r,s) = (b_m(r,s),b_m(q,p))\mbox{-entry of } (A_m)^n, $$
where $A_m$ is the $2^m \! \times \! 2^m$ matrix recursively defined by 
$$ A_{k} = \begin{bmatrix} 
\Small{\begin{bmatrix} A_{k-2} & \mathbb{O}_{k-2} \\ \mathbb{O}_{k-2} & \mathbb{O}_{k-2} \end{bmatrix}}
& A_{k-1} \\ A_{k-1} & \mathbb{O}_{k-1} \end{bmatrix} $$
for $k=2, \dots, m$, with seed matrices
$A_0 = \begin{bmatrix} 1 \end{bmatrix}$ and \
$A_1 = \begin{bmatrix} 0 & 1 \\ 1 & 0 \end{bmatrix}$.
Here 
$$ b_m(r,s) = \frac{2}{3} (2^{m-r+2[\frac{r}{2}]} - 2^{m-r}) + \frac{1}{3} (2^{s} - 2^{s-2[\frac{s}{2}]}) + 1. $$
\end{theorem}

\begin{proof}
It is worthwhile mentioning that the enumeration of domino tilings of this special case 
$\mathbb{A}_{m \times n}(1,1,1,1)$ is answered as $G_{m \times n}(0,1,1)$ in Theorem~\ref{thm:dimer}.

Domino tilings of a region (or equivalently, pure dimer coverings on the related square lattice) 
is known to be very sensitive to its boundary condition.
As an evidence, if it has a non-trivial boundary state 
in some part of which letters {\texttt a} and {\texttt b} appear in turn as in Figure~\ref{fig:boundary},
then not only boundary but some interior squares must be covered by dominos in the unique way
as the shaded region in the figure.

\begin{figure}[h]
\includegraphics{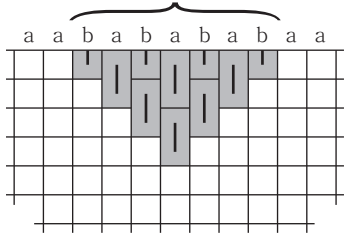}
\caption{Non-trivial boundary state condition}
\label{fig:boundary}
\end{figure}

Using this trick, by letting the bottom and top states as in Figure~\ref{fig:dominomosaic},
we can cover the related four triangular corners by dominos in the unique way
so that the remaining set of squares is the Aztec octagon we consider.
The bottom state of length $m$ consists of three parts; 
$r \! - \! 1$ letters of {\texttt a} and {\texttt b} in turn ending with {\texttt b},
$m \! - \! r \! - \! s \! + \! 2$ letters of only {\texttt a}, and
followed by $s \! - \! 1$ letters of {\texttt a} and {\texttt b} in turn beginning with {\texttt b}.
This is the $b_m(r,s)$th state among $2^m$ states, where for $r, s \geq 2$, 
\begin{equation*} 
\begin{split}
b_m(r,s) = \ & 2^{m-1} (\mbox{or } 2^{m-2}) + \cdots + 2^{m-r+3} + 2^{m-r+1} \hspace{8mm} \\
                 & + 2^{s-2} + 2^{s-4} + \cdots + 2^1 (\mbox {or } 2^0) + 1, 
\end{split}
\end{equation*}
where the choice of $2^{m-1}$ or $2^{m-2}$ (similarly $2^0$ or $2^1$) depends on 
whether $r$ (respectively $s$) is even or odd.
Further denote that 
\begin{equation*} 
\begin{split}
b_m(1,1) = \ & 1, \\
b_m(1,s) = \ & 2^{s-2} + 2^{s-4} + \cdots + 2^1 (\mbox {or } 2^0) + 1, \\
b_m(r,1) = \ & 2^{m-1} (\mbox{or } 2^{m-2}) + \cdots + 2^{m-r+3} + 2^{m-r+1} + 1.
\end{split}
\end{equation*}
Remember that the bottom and top states are obtained by reading off letters from right to left. 
This number $b_m(r,s)$ can be re-defined as written in the theorem.

\begin{figure}[h]
\includegraphics{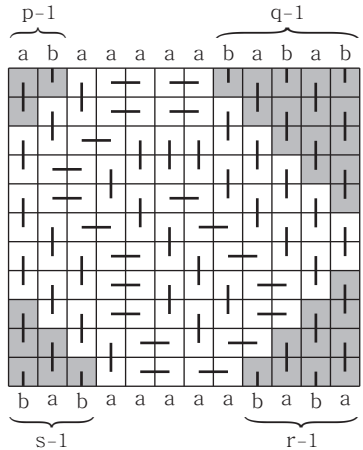}
\caption{Suitably adjacent $12 \! \times \! 11$-mosaic associated to the left picture in Figure~\ref{fig:domino}}
\label{fig:dominomosaic}
\end{figure}

As in the proof of Theorem~\ref{thm:dimer},
the $(b_m(r,s),b_m(q,p))$-entry of $(A_m)^n$ after applying $x \! = \! y \! = \! 1$ and $v \! = \! 0$
is the number of suitably adjacent $m \! \times \! n$-mosaics associated to the triple 
$\langle {\texttt a} \cdots {\texttt a}, \epsilon^m_{b_m(r,s)}, \epsilon^m_{b_m(q,p)} \rangle$.
This completes the proof.
\end{proof}

Lastly, we mention the enumeration problem of domino tilings of an Aztec octagon with holes 
as the right picture in Figure~\ref{fig:domino}.
Let $S$ be a set of squares in $\mathbb{A}_{m \times n}(p,q,r,s)$.
Let $\alpha_{m \times n}(S;p,q,r,s)$ denote the number of domino tilings of $\mathbb{A}_{m \times n}(p,q,r,s)$ 
restricting all squares of $S$ removed.

\begin{theorem} \label{thm:aztecholes}
For domino tilings of an Aztec octagon $\mathbb{A}_{m \times n}(p,q,r,s)$ with a set $S$ of holes, 
$$ \alpha_{m \times n}(S;p,q,r,s) = (b_m(r ,s),b_m(q,p))\mbox{-entry of } \prod^n_{i=1} A_{m,i}, $$
where $A_{m,i}$ is recursively defined in Theorem~\ref{thm:fixed}
and $b_m(r,s)$ is in Theorem~\ref{thm:aztec}.
\end{theorem}

\begin{proof}
Theorem~\ref{thm:fixed} combined with Theorem~\ref{thm:aztec} guarantees the theorem.
\end{proof}

\end{document}